\documentclass[a4paper]{elsarticle}

\usepackage{amssymb}
\usepackage{amsmath}
\usepackage{bbm}
\usepackage{environ}
\usepackage{multirow}
\usepackage{longtable}
\usepackage{comment}
\usepackage{placeins}
\usepackage{mathtools}
\usepackage{enumitem}
\usepackage[utf8]{inputenc}
\usepackage{algpseudocode}
\usepackage{algorithm}
\usepackage{array}
\usepackage[pdfencoding=auto,psdextra]{hyperref}
\usepackage{booktabs}
\usepackage{bookmark}
\usepackage{hypcap} 
\evensidemargin\oddsidemargin
\usepackage{graphicx}
\usepackage{xcolor}

\usepackage{color}
\usepackage{epsfig}
\usepackage{caption}
\usepackage{subcaption}
\usepackage{amsthm}

\usepackage[draft,nomargin,inline]{fixme}
\fxsetface{inline}{\itshape}
\fxsetface{env}{\itshape}
\fxusetheme{color}
\usepackage{url}
\urldef{\mailsa}\path|{djukanovic, raidl}@ac.tuwien.ac.at,|
\urldef{\mailsb}\path|christian.blum@iiia.csic.es|

\usepackage{tikz}
\usetikzlibrary{positioning}
\definecolor{canaryyellow}{rgb}{1.0, 0.94, 0.0}
\definecolor{brightgreen}{rgb}{0.4, 1.0, 0.0}
\definecolor{jazzberryjam}{rgb}{0.65, 0.04, 0.37}

\newtheorem{thm}{Theorem}

\newtheorem{rem}{Remark}

\begin{document}
	
	
	\title{The Signed (Total) Roman Domination Problem on some Classes of Planar Graphs -- Convex Polytopes}
	
	\author[1]{Tatjana Zec }
	\author[1]{Marko Djukanovi\'c }
	\author[1]{Dragan Mati\'c}
	
	 \address[1]{$\{tatjana.zec|marko.djukanovic|dragan.matic\}@pmf.unibl.org$,\\   Faculty of Natural Sciences and Mathematics, University of Banja Luka, Bosnia and Herzegovina}
	\begin{abstract}
		In this paper we deal with the calculation of the signed (total) Roman domination numbers, $\gamma_{sR}$ and $\gamma_{stR}$ respectively, on a few classes of planar graphs from the literature. We give proofs for the exact values of the numbers $\gamma_{sR}(A_n)$ and $\gamma_{sR}(R_n)$ as well as the numbers 
		$\gamma_{stR}(S_n)$ and $\gamma_{stR}(T_n)$.  For some other classes of planar graphs, such as  $Q_n$, 
		 and $T_n''$, lower and upper bounds on  $\gamma_{sR}$ are calculated and proved. 
	\end{abstract}
	\maketitle

	\section{Introduction}\label{sec:introduction}
	
	Let $G = (V,E)$ be a finite, undirected graph without cycles, where $V$ denotes the set of vertices and $E$ denotes the set of edges.
	The $open$ neighborhood of vertex $i$ is represented by the set $N(i)=\{j \in V \mid ij \in E\}$, while
	the $closed$ neighborhood is represented by $N[i] = N(i) \cup \{i\}$.
	
	\emph{The signed Roman domination problem} (SRDP) was  introduced in \citep{ahangar2014signed}.
	In the SRDP, the dominating function $f:V \mapsto \{-1,1,2\}$  satisfies the following two conditions:
	\begin{equation}
		s(v) = \sum_{u \in N[v]}{f(u)} \geqslant 1,\quad \forall v \in V,
		\label{eq:c2srdp}
	\end{equation}
	and 
	
	\begin{equation}
		f(v)=-1 \implies ((\exists j)\;j \in N(v) \land f(j)=2), \quad \forall v \in V
		\label{eq:c1}
	\end{equation}
	
	In other words, for each vertex of $G$ the sum of the values assigned to a vertex and its neighbors is at least 1 and for every vertex $v\in V$ for which $f(v)=-1$, vertex $v$ must be adjacent to at least one vertex $j$ for which $f(j)=2$. Such a function $f$ is called signed Roman domination (SRD) function.
	
	In the sub-variant of SRDP, called the \emph{Signed total Roman domination problem} (STRDP), introduced in \citep{volkmann2016signed1}, the dominating function $f:V \mapsto \{-1,1,2\}$  has to satisfy the slightly different condition (\ref{eq:c2srdp}): for each vertex the sum of the values assigned exactly to its neighbors is at least 1. Formally, 
	\begin{equation}
		s(v)^{tot}=\sum_{u \in N(v)}{f(u)} \geq 1,\quad v \in V.
		\label{eq:c2strdp}
	\end{equation}
	
	The condition (\ref{eq:c1}) remains the same, i.e. for every vertex $v\in V$ for which $f(v)=-1$ there must exists at least one vertex $j$ adjacent to $v$, satisfying $f(j)=2$.
	
	Analogously to the case od SRDP, one can define the signed total Roman domination (STRD) function, that is the function which satisfied the condition described for the STRDP.

	The weight of an SRD function $f$ on the set of vertices $V$ of the graph $G$ is defined by $f(V) = \sum_{v\in V}f(v)$. The signed (total)
	Roman domination number, denoted $\gamma_{stR}$, is the minimum weight of an S(T)RD function in the considered graph $G$.

	Please notice that an SRD (STRD) function $f$ induces a partition of the vertices $V$, i.e., $V = (V_{-1}, V_1, V_2)$, where $V_i = \{v\in  V : f(v) = i\}$.
	
	From there, it holds $f(V)=\sum_{v\in V}f(v) = - |V_{-1}| + |V_1| + 2 \cdot |V_2|$. 
	
	\subsection{Previous work}
	The basic problem in this field, the Roman domination problem (RDP), was introduced by Cockayne et al. in~\cite{cockayne2004roman}. Since then, it was atracted by many scientists to theoretically study   this problem on many special classes of graphs and graph structures, see~\cite{ramezani2016roman,xueliang2009roman,liu2012upper,liu2012roman,rad2011roman,sheikholeslami2011roman,pushpam2012roman,liu2021roman}, as well as algorithmically, see~\cite{ivanovic2016improved,khandelwal2021roman,liedloff2008efficient}.  RDP problem for some classes of planar graphs called convex polytopes is studied in~\cite{kartelj2021roman}. 
	Ahangar et al.  introduced the SRDP in \citep{ahangar2014signed} and presented several upper and lower bounds on this problem. The most known lower bound for the general graphs for SRDP is given by
	\begin{equation}\label{eq:lower_srdp2}
	     \gamma_{sR}(G) \geqslant \frac{-2 \Delta^2 + 2\Delta \delta + \Delta + 2 \delta + 3}{ (\Delta + 1) (2 \Delta + \delta + 3)} n
    \end{equation} 
 
	where $n$, $\delta$ and $\Delta$ represent  the number of vertices, the minimum and maximum degree among the vertices of $G$, respectivelly.  
	
	The exact value of the SRD number is known for some special classes of graphs, such as stars, complete graphs, cycles, paths and complete bipartite graphs.  SRDP problem is known to be $\mathcal{NP}$--complete even when restricted to bipartite and planar graphs~\citep{shao2017signed}.  The concept of SRD problem on digraphs is proposed in~\cite{sheikholeslami2015signed}. The SRD problem with respect to the operation of graph join is studied in~\cite{behtoei2014signed}.  A more general variant of SRDP, called the \emph{Signed Roman $k$-domination problem} (SRkDP) was analyzed in \citep{henning2016signedk} and \citep{amjadi2020signed}, and specially on trees in~\cite{henning2015signed}. 
	In the general case, the number 1 on the right-hand side of inequality (\ref{eq:c2srdp}) is replaced by an arbitrary positive integer $k$. There are a few variants of SRD problem proposed in the literture that arise from the theoretical interests: the weak signed Roman domination problem~\cite{volkmann2020weak}, the signed Roman edge domination problem ~\cite{ahangar2016signed}, among others.

	Following the ideas in \citep{ahangar2014signed}, Volkmann introduced a study of the  STRD number in \citep{volkmann2016signed2,volkmann2016signed1}, presenting different
	bounds and exact values on several classes of graphs. 
		Several bounds for the STRD number on general graphs are shown in~\cite{maksimovic2018some}. The tightest of them on the classes of graph considered in this work are\\ 
\begin{enumerate}
\item if $\delta<\Delta$:
		\begin{equation}\label{eq:lower_strdp1}
              \gamma_{stR}(G) \geqslant \left\lceil\frac{(2 \delta + 3 - 2 \Delta) n}{2 \Delta + \delta}\right\rceil
		\end{equation}
\item if $\delta \geqslant 3$:
	\begin{equation}\label{eq:upper_strdp1}
		 \gamma_{stR}(G) \leqslant n-1
\end{equation}
\end{enumerate}

	The most of the other lower bounds for STRD number from the literature assimptotically behave like $-n$ when $ n$ grows for the considered graphs in this paper and, therefore, are  dominated by the above bounds.  

	The STRD numbers of complete bipartite graphs and wheels is considered in~\cite{yan2017signed}. The concept of STRD numbers on digraphs is studied in~\cite{volkmann2017signed}. Total Roman domination number of rooted product graphs is presented in ~\cite{cabrera2020total}. Relations between STRD problem and domatic numbers are studied in ~\cite{volkmann2016signed2}. There are many variants of STRD problems studied from the theoretical point of view: the non-negative STRD problem~\cite{dehgardi2020nonnegative}, the signed total Roman $k$-domination problem~\cite{dehgardi2016signed}, the signed total Roman edge domination problem~\cite{asgharsharghi2017signed}, the signed total double Roman domination problem~\cite{shahbazi2020bounds}, the quasi total Roman domination problem~\cite{garcia2019quasi}, among others.  
	
	For a more detailed literature review about SRDP and STRDP, we refer interesting readers to \citep{chellali2020varieties}.

	\section{Main results}
	
	\subsection{Signed Roman domination number for $A_n$}
	
	This class of planar graphs is introduced in~\cite{imran2013classes}. 
	It is defined as a graph $A_n=(V(A_n), E(A_n))$ in the following way: 
	$$V(A_n) = \{a_i, b_i, c_i \mid i=0,\ldots, n-1\}$$ and 
	$$E(A_n) = \{a_i a_{i+1}, b_i b_{i+1,} c_i c_{i+1}, a_i b_i, b_i c_i, a_{i+1}b_i, b_{i+1}c_i \mid i=0,\ldots, n-1 \}.$$ One such graph is displayed in Figure~\ref{fig:an}.
	
	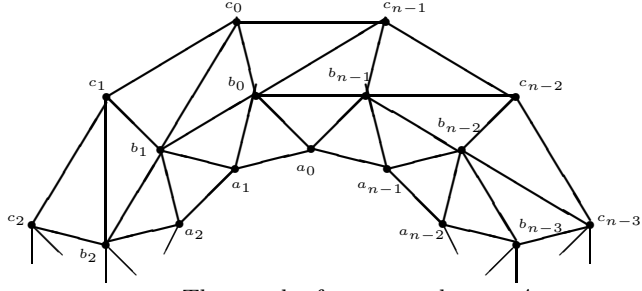
\begin{figure}[htbp]
    \centering\setlength\unitlength{1mm}

\setlength\unitlength{1mm}
\begin{picture}(130,35)
\thicklines
\tiny
\put(47.7,5.0){\circle*{1}} \put(47.7,5.0){\line(1,1){7.3}} \put(47.7,5.0){\line(-4,-1){9.7}} \put(47.7,5.0){\line(-1,4){2.5}} 
\put(55.0,12.3){\circle*{1}} \put(55.0,12.3){\line(4,1){10.0}} \put(55.0,12.3){\line(-4,1){9.8}} \put(55.0,12.3){\line(1,4){2.8}} 
\put(65.0,15.0){\circle*{1}} \put(65.0,15.0){\line(4,-1){10.0}} \put(65.0,15.0){\line(-1,1){7.2}} \put(65.0,15.0){\line(1,1){7.2}} 
\put(75.0,12.3){\circle*{1}} \put(75.0,12.3){\line(1,-1){7.3}} \put(75.0,12.3){\line(-1,4){2.8}} \put(75.0,12.3){\line(4,1){9.8}} 
\put(82.3,5.0){\circle*{1}} \put(82.3,5.0){\line(1,4){2.5}} \put(82.3,5.0){\line(4,-1){9.7}} 
\put(38.0,2.2){\circle*{1}} \put(38.0,2.2){\line(3,5){7.2}} \put(38.0,2.2){\line(0,1){19.6}} \put(38.0,2.2){\line(-4,1){9.7}} 
\put(45.2,14.8){\circle*{1}} \put(45.2,14.8){\line(5,3){12.6}} \put(45.2,14.8){\line(3,5){10.0}} \put(45.2,14.8){\line(-1,1){7.1}} 
\put(57.8,22.0){\circle*{1}} \put(57.8,22.0){\line(1,0){14.5}} \put(57.8,22.0){\line(5,3){17.1}} \put(57.8,22.0){\line(-1,4){2.6}} 
\put(72.2,22.0){\circle*{1}} \put(72.2,22.0){\line(5,-3){12.6}} \put(72.2,22.0){\line(1,0){19.6}} \put(72.2,22.0){\line(1,4){2.6}} 
\put(84.8,14.8){\circle*{1}} \put(84.8,14.8){\line(3,-5){7.2}} \put(84.8,14.8){\line(5,-3){16.9}} \put(84.8,14.8){\line(1,1){7.1}} 
\put(92.0,2.2){\circle*{1}} \put(92.0,2.2){\line(4,1){9.7}} 
\put(28.3,4.8){\circle*{1}} \put(28.3,4.8){\line(3,5){9.8}} 
\put(38.1,21.9){\circle*{1}} \put(38.1,21.9){\line(5,3){17.0}} 
\put(55.2,31.7){\circle*{1}} \put(55.2,31.7){\line(1,0){19.7}} 
\put(74.8,31.7){\circle*{1}} \put(74.8,31.7){\line(5,-3){17.0}} 
\put(91.9,21.9){\circle*{1}} \put(91.9,21.9){\line(3,-5){9.8}} 
\put(101.7,4.8){\circle*{1}}
\thinlines
\put(47.7,5.0){\line(-1,-2){2}} \put(82.3,5.0){\line(1,-2){2}}
\put(38.0,2.2){\line(1,-1){4}} \put(92.0,2.2){\line(-1,-1){4}}
\put(38.0,2.2){\line(0,-1){5}} \put(92.0,2.2){\line(0,-1){5}}
\put(28.3,4.8){\line(0,-1){5}} \put(101.7,4.8){\line(0,-1){5}}
\put(28.3,4.8){\line(1,-1){4}} \put(101.7,4.8){\line(-1,-1){4}}
\put(48.3,3.5){$a_2$} \put(54.5,9.7){$a_1$} \put(63.0,12.0){$a_0$} \put(71,9.7){$a_{n-1}$} \put(76.5,3.5){$a_{n-2}$} 
\put(34.5,0.5){$b_2$} \put(41.2,14.5){$b_1$} \put(54.0,23.3){$b_0$} \put(67.2,24.5){$b_{n-1}$} \put(81.6,17.8){$b_{n-2}$} \put(92.3,4.5){$b_{n-3}$} 
\put(25,5.4){$c_2$} \put(35.7,23.3){$c_1$} \put(53.6,33.6){$c_0$} \put(74.4,33.6){$c_{n-1}$} \put(92.3,23.3){$c_{n-2}$} \put(102.6,5.4){$c_{n-3}$} 

\end{picture}
    \caption{\label{fig:an} \small The graph of convex polytope $A_n$}
\end{figure}
	As it can be seen from Figure~\ref{fig:an}, the degree of every $a$ vertex is four. More precisely, vertex $a_{i}$ is adjacent with two $a$ vertices, i.e. $a_{i-1}$ and $a_{i+1}$, and two $b$ vertices $b_{i-1}$ and $b_i$. Similar, every $c$ vertex has also degree four. Vertex $c_i$ is connected by edges with $b_i$, $b_{i+1}$, $c_{i-1}$ and $c_{i+1}$. The arbitrary $b$ vertex has six neighbors. Vertex $b_i$ is adjacent with $a_i$, $a_{i+1}$, $b_{i-1}$, $b_{i+1}$, $c_{i-1}$ and $c_i$. We used these properties in the proof of Theorem 1.
	
	Please note that vertex indices are taken modulo $n$ throughout the
	whole paper. This means that vertex $v_{i}$ for any $i \in \mathbb{N}$ represents always the vertex $v_{i \mod n}$.  
	\begin{thm}
		Let $n\geq 5$.	$\gamma_{sR}(A_n) = 0$.

	\end{thm}
\begin{proof}
	
	{\em Step 1. \underline{Upper bound}}
		
	Let us
	define a function $f:V \mapsto \{-1,1,2\}$ by the partition  $(V_{-1},V_1,V_2)$ of the set $V(A_n)$. Let $V_2 = \{b_{i} \mid i=0, \ldots ,n-1\}$, $V_1 = \emptyset$ and $V_0=\{a_i,c_i \mid i=0, \ldots ,n-1\}$. 
	Then, $f(V(A_n)) = 2|V_2| + (-1)|V_{-1}| = 2n-2n = 0.$   Let us prove that such $f$ is an SRD function.
	
	From the definition of the function $f$, a vertex from $V_{-1}$ is either $a$ vertex or $c$ vertex. In the graph $A_n$, each $a$ vertex and each $c$ vertex are adjacent with two $b$ vertices. Since every $b$ vertex is labeled with 2, it holds that each vertex from  $V_{-1}$ has its neighbor from $V_2$.

	Now we prove that, for each vertex, condition (\ref{eq:c2srdp}) is satisfied.
	
	Let $a_i$ be an arbitrary $a$ vertex.
	
	$$s(a_i) = f(a_{i-1})+f(a_i)+f(a_{i+1})+f(b_{i-1})+f(b_i) = -1-1-1+2+2 = 1$$
	
	Similarly, for an arbitrary $c_i$, it holds 
	
	$$s(c_i) = f(c_{i-1})+f(c_i)+f(c_{i+1})+f(b_{i})+f(b_{i+1}) = -1-1-1+2+2 = 1$$

	Let $b_i$ be an arbitrary $b$ vertex. It holds
	
	\begin{align*}
		s(b_i) &= f(a_i)+f(a_{i+1})+f(b_{i-1}) +f(b_i) +f(b_{i+1}) + f(c_{i-1})+f(c_i) \\
		&=-1-1+2+2+2-1-1=2
	\end{align*} 
	We proved that $\gamma_{sR}(A_n) \leqslant 0$.
	
	{\em Step 2. \underline{Lower bound}}
	
	Now we prove that $\gamma_{sR}(A_n) \geqslant 0$. Let $\overline{f}$ be an arbitrary SRD function. We will prove that $	\overline{f}(V(A_n)) \geqslant 0$.
	
	For each vertex $a_i$  from the structure of graph $A_n$ we get 
	
	$$ \overline{f}(a_{i-1})+\overline{f}(a_i)+\overline{f}(a_{i+1})+\overline{f}(b_{i-1})+\overline{f}(b_i) \geqslant 1$$
	
	Summing these inequalities for each $i$, we get
	
	\begin{equation} 3\sum_{i=0}^{n-1}\overline{f}(a_i)+2 \sum_{i=0}^{n-1}\overline{f}(b_i)\geqslant n
		\label{eqn:Aa}\end{equation}
	
	For each vertex $c_i$, from the structure of polytope $A_n$ we get 
	
	\begin{equation} \overline{f}(c_{i-1})+\overline{f}(c_i)+\overline{f}(c_{i+1})+\overline{f}(b_{i})+\overline{f}(b_{i+1}) \geqslant 1 
		\label{eqn:Ac}\end{equation}
	
	Summing up these inequalities for each $i$, we get
	
	\begin{equation} 3\sum_{i=0}^{n-1}\overline{f}(c_i)+2 \sum_{i=0}^{n-1}\overline{f}(b_i)\geqslant n
	\end{equation}
	
	When we sum up (\ref{eqn:Aa}) and (\ref{eqn:Ac}) we obtain
	
	\begin{equation}
		3\overline{f}(V(A_n)) +\sum_{i=0}^{n-1}\overline{f}(b_i) \geqslant 2n
	\end{equation}
	
	Since $\sum_{i=0}^{n-1}\overline{f}(b_i)\leqslant 2n$ holds, we get 
	$	3\overline{f}(V(A_n)) \geqslant 0$, i.e.
	
	$$	\overline{f}(V(A_n)) \geqslant 0,$$
	
	which concludes the proof.
\end{proof}
	
\begin{rem}
Consider the function obtained by the following redefinition of the function which occurs in the basic RDP solved in~\cite{kartelj2021roman}: those vertices which are there  labeled by $0$ we redefine  by  label $-1$ in case of SRDP. One could conclude that  function defined in this way does not satisfy    condition (\ref{eq:c2srdp}), i.e. this solution is not an admissible SRDF for $A_n$. Therefore, the (optimal) solution of RDP could not be directly used to get an (optimal) solution of SRDP.
\end{rem}

 \begin{rem}	
 	We now look back on the lower bound given by (\ref{eq:lower_srdp2}).  For graph $A_n$ we easily see that  $\Delta = 6$ and $\delta = 4$. According  to (\ref{eq:lower_srdp2}) we get $\gamma_{sR}(A_n) \geqslant-\frac{3}{19}n$. Therefore, this result  could not  be directly used to prove the 
 	previous theorem.
 \end{rem}
	
	\subsection{Signed Roman domination number for $R_n$}
	This class of planar graphs is introduced in~\cite{baca1992magic}. 
	It is defined as a graph $R_n=(V(R_n), E(R_n))$ in the following way: 
	$$V(R_n) = \{a_i, b_i, c_i \mid i=0,\ldots, n-1\}$$ and 
	$$E(R_n) = \{a_i a_{i+1}, b_i b_{i+1,} c_i c_{i+1}, a_i b_i, b_i c_i, a_{i+1}b_i \mid i=0,\ldots, n-1 \}.$$ An exemplar of such graph is displayed in Figure~\ref{fig:rn}.
	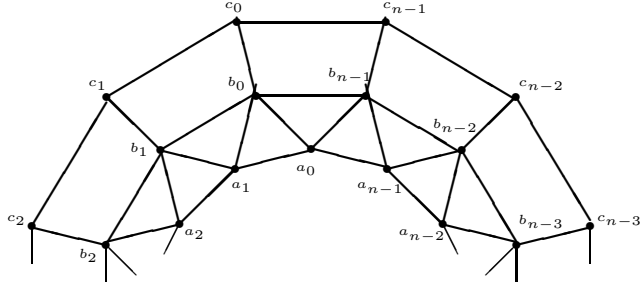
\begin{figure}[htbp]
    \centering\setlength\unitlength{1mm}

\setlength\unitlength{1mm}
\begin{picture}(130,36)
\thicklines
\tiny
\put(47.7,5.0){\circle*{1}} \put(47.7,5.0){\line(1,1){7.3}} \put(47.7,5.0){\line(-4,-1){9.7}} \put(47.7,5.0){\line(-1,4){2.5}} 
\put(55.0,12.3){\circle*{1}} \put(55.0,12.3){\line(4,1){10.0}} \put(55.0,12.3){\line(-4,1){9.8}} \put(55.0,12.3){\line(1,4){2.8}} 
\put(65.0,15.0){\circle*{1}} \put(65.0,15.0){\line(4,-1){10.0}} \put(65.0,15.0){\line(-1,1){7.2}} \put(65.0,15.0){\line(1,1){7.2}} 
\put(75.0,12.3){\circle*{1}} \put(75.0,12.3){\line(1,-1){7.3}} \put(75.0,12.3){\line(-1,4){2.8}} \put(75.0,12.3){\line(4,1){9.8}} 
\put(82.3,5.0){\circle*{1}} \put(82.3,5.0){\line(1,4){2.5}} \put(82.3,5.0){\line(4,-1){9.7}} 
\put(38.0,2.2){\circle*{1}} \put(38.0,2.2){\line(3,5){7.2}} \put(38.0,2.2){\line(-4,1){9.7}} 
\put(45.2,14.8){\circle*{1}} \put(45.2,14.8){\line(5,3){12.6}} \put(45.2,14.8){\line(-1,1){7.1}} 
\put(57.8,22.0){\circle*{1}} \put(57.8,22.0){\line(1,0){14.5}} \put(57.8,22.0){\line(-1,4){2.6}} 
\put(72.2,22.0){\circle*{1}} \put(72.2,22.0){\line(5,-3){12.6}} \put(72.2,22.0){\line(1,4){2.6}} 
\put(84.8,14.8){\circle*{1}} \put(84.8,14.8){\line(3,-5){7.2}} \put(84.8,14.8){\line(1,1){7.1}} 
\put(92.0,2.2){\circle*{1}} \put(92.0,2.2){\line(4,1){9.7}} 
\put(28.3,4.8){\circle*{1}} \put(28.3,4.8){\line(3,5){9.8}} 
\put(38.1,21.9){\circle*{1}} \put(38.1,21.9){\line(5,3){17.0}} 
\put(55.2,31.7){\circle*{1}} \put(55.2,31.7){\line(1,0){19.7}} 
\put(74.8,31.7){\circle*{1}} \put(74.8,31.7){\line(5,-3){17.0}} 
\put(91.9,21.9){\circle*{1}} \put(91.9,21.9){\line(3,-5){9.8}} 
\put(101.7,4.8){\circle*{1}}
\thinlines
\put(47.7,5.0){\line(-1,-2){2}} \put(82.3,5.0){\line(1,-2){2}}
\put(38.0,2.2){\line(1,-1){4}} \put(92.0,2.2){\line(-1,-1){4}}
\put(38.0,2.2){\line(0,-1){5}} \put(92.0,2.2){\line(0,-1){5}}
\put(28.3,4.8){\line(0,-1){5}} \put(101.7,4.8){\line(0,-1){5}}
\put(48.3,3.5){$a_2$} \put(54.5,9.7){$a_1$} \put(63.0,12.0){$a_0$} \put(71,9.7){$a_{n-1}$} \put(76.5,3.5){$a_{n-2}$} 
\put(34.5,0.5){$b_2$} \put(41.2,14.5){$b_1$} \put(54.0,23.3){$b_0$} \put(67.2,24.5){$b_{n-1}$} \put(81,17.8){$b_{n-2}$} \put(92.3,5.0){$b_{n-3}$} 
\put(25,5.4){$c_2$} \put(35.7,23.3){$c_1$} \put(53.6,33.6){$c_0$} \put(74.4,33.6){$c_{n-1}$} \put(92.3,23.3){$c_{n-2}$} \put(102.6,5.4){$c_{n-3}$} 

\end{picture}
    \caption{\label{fig:rn} \small The graph of convex polytope $R_n$}
\end{figure}
	
	Comparing the polytopes $A_n$ and $R_n$, it can be seen that the connections between $a$ and $b$ vertices in these two polytopes are the same. Unlike the connections between $b$ and $c$ vertices in $A_n$, vertex $c_i$ is not adjacent with vertex $b_{i+1}$ in polytope $R_n$.  We used information about adjacent vertices in the proof of Theorem 2.

	\begin{thm} \label{thm:rn_exact}
		$\gamma_{sR}(R_{3k})=2k$, $2k+1\leqslant \gamma_{sR}(R_{3k+1})\leqslant 2k+2$ and $\gamma_{sR}(R_{3k+2}) = 2k+2$.
		
	\end{thm}
	
\begin{proof}
	{\em Step 1. \underline{Upper bound}}\\
	Case 1: $n = 3k$.\\
	Let function $f$ be defined as follows:
	
	$V_2 = \{b_{i} \mid i=0, \ldots ,3k-1\}$,  $V_1 = \{c_{3i} \mid i=0,\ldots, k-1\}$, and $V_{-1} = \{a_{i} \mid i=0,\ldots,3k-1\}\cup\{c_{3i+1},c_{3i+2} \mid i=0, \ldots,k-1\}$. Then, it holds
	$$f(V(R_{3k})) = 2|V_2| +|V_1|+ (-1) \cdot |V_{-1}| = 2(3k)+k-(3k+2k) = 2k.$$

	Let us prove that $f$ is an SRD function. 
	
	The set $V_{-1}$ contains only $a$ and $c$ vertices. From the definition of graph $R_n$, each $a$ vertex is adjacent to two $b$ vertices, which are labeled by 2 and each $c$ vertex is adjacent with one $b$ vertex. Since all $b$ vertices are labeled by 2, we showed that each vertex labeled by -1 has at least one adjacent vertex labeled by 2. Therefore, the condition (\ref{eq:c1}) is satisfied.
	
	Let us prove that the condition (\ref{eq:c2srdp}) is satisfied.
	
	Let $a_i$ be an arbitrary $a$ vertex. 
	$$s(a_i) = f(a_{i-1})+f(a_i)+f(a_{i+1})+f(b_{i-1})+f(b_i) = -1-1-1+2+2 = 1$$
	
	In order to prove that the condition (\ref{eq:c2srdp}) holds for each $b_i$ and $c_i$ vertices, $i = 0,\ldots  3k-1$, we analyze three subcases:
	
	Subcase 1. $i = 3l$
	\begin{align*}
		s(b_i) &= s(b_{3l})= f(b_{3l-1})+f(b_{3l})+f(b_{3l+1})+f(a_{3l})+f(a_{3l+1})+f(c_{3l}) \\
		&= 2+2+2-1-1+1 = 5
	\end{align*}
	$$s(c_i) =s(c_{3l})= f(c_{3l-1})+f(c_{3l})+f(c_{3l+1})+f(b_{3l}) = -1+1-1+2 = 1$$

	Subcase 2.  $i = 3l+1$
	\begin{align*}
		s(b_i) &=s(b_{3l+1})= f(b_{3l})+f(b_{3l+1})+f(b_{3l+2})+f(a_{3l+1})+f(a_{3l+2})+f(c_{3l+1}) \\
		& = 2+2+2-1-1-1 = 3
	\end{align*}
	$$s(c_i) =s(c_{3l+1})= f(c_{3l})+f(c_{3l+1})+f(c_{3l+2})+f(b_{3l+1}) = 1-1-1+2 = 1$$
	
	Subcase 3.  $i = 3l+2$
	\begin{align*}
		s(b_i) &=s(b_{3l+2})= f(b_{3l+1})+f(b_{3l+2})+f(b_{3l+3})+f(a_{3l+2})+f(a_{3l+3})+f(c_{3l+2}) \\
		&= 2+2+2-1-1-1 = 3
	\end{align*}
	$$s(c_i) =s(c_{3l+2})= f(c_{3l+1})+f(c_{3l+2})+f(c_{3l+3})+f(b_{3l+2}) = -1-1+1+2 = 1$$
	
	Since the conditions  (\ref{eq:c2srdp}) and (\ref{eq:c1}) are satisfied for each vertex, we proved that  $f$ is an SRD function  and $\gamma_{sR}(R_{3k})\leqslant 2k$.
	
	Case 2: $n = 3k+1$.

	We define the function $f$ similarly as in Case 1.
	
	$V_2 = \{b_{i} \mid i=0, \ldots ,3k\}$ and     
	$V_1 = \{c_{3i} \mid i=0,\ldots, k\}$, while the rest of $c$--vertices and all $a$--vertices  belong to $V_{-1}$. More precisely,  $V_{-1} = \{a_{i} \mid i=0,\ldots,3k\}\cup\{c_{3i+1},c_{3i+2} \mid i=0, \ldots,k-1\}$. Then, it holds
	$$f(V(R_{3k+1})) = 2|V_2| +|V_1|+ (-1) \cdot |V_{-1}| = 2(3k+1)+k+1-(3k+1+2k) = 2k+2.$$ 
	
	Proof that  conditions  (\ref{eq:c2srdp}) and (\ref{eq:c1})  are satisfied is similar as in Case 1, so it is omitted here. Therefore,  $\gamma_{sR}(R_{3k+1})\leqslant 2k+2$.

	Case 3. $n = 3k+2$. 
	
	As in previous two cases, we introduce the function $f$: 
	$V_2 = \{b_{i} \mid i=0, \ldots ,3k+1\}$ and     
	$V_1 = \{c_{3i} \mid i=0,\ldots, k\}$. The rest of $c$ vertices and all $a$ vertices  belong to $V_{-1}$, i.e. 
	
	$V_{-1} = \{a_{i} \mid i=0,\ldots,3k+1\}\cup\{c_{3i+1},c_{3i+2} \mid i=0, \ldots,k-1\}\cup \{c_{3k+1}\}$. Then, it holds
	$$f(V(R_{3k+2})) = 2|V_2| +|V_1|+ (-1) \cdot |V_{-1}| = 2(3k+2)+k+1-(3k+2+2k+1) = 2k+2.$$ 
	
	Proof that $f$ is SRD function is similar to the previous two cases. 
	
	

	So, 
	$\gamma_{sR}(R_{3k+1})\leqslant 2k+2$.

	Step 2. \emph{\textit{Lower bound}.} 
	
	Let $\overline{f}$ be a SRDP function.
	Similar as in the case of polytope $A_n$, we exploit the structure of graph $R_n$. Concerning any $c_i$ vertex w.r.t.\  condition (\ref{eq:c2strdp}) it holds that 
	\begin{equation}
		s(c_i ) = \overline{f}(c_i) + \overline{f}(b_i) + \overline{f}(c_{i-1}) + \overline{f}(c_{i+2}) \geqslant 1.
		\label{eq:r_n_i} 
	\end{equation}
	Summing up (\ref{eq:r_n_i} ) over all $i=0,\ldots, n$, we come up with the following inequality
	
	\begin{equation}
		3 \sum_{i=0}^{n-1} \overline{f}(c_i) + \sum_{i=0}^{n-1}\overline{f} (b_i) \geqslant n.
		\label{eq:rn_sum_c}
	\end{equation}
	Similarly, we do with $a$-vertices, that is, for arbitrary $a_i$ we have 
	\begin{equation}
		\overline{f}(a_i) + \overline{f}(a_{i-1}) + \overline{f}(a_{i-2}) + \overline{f}(b_{i-1}) + \overline{f}(b_i) \geqslant 1.
		\label{eq_r_n_i2}
	\end{equation}
	By summing up (\ref{eq_r_n_i2}) for all $i=0,\ldots,n-1$, we have 
	\begin{equation}
		3 \sum_{i=0}^{n-1} \overline{f}(a_i) + 2\sum_{i=0}^{n-1}\overline{f} (b_i) \geqslant n.
		\label{eq:rn_sum_a}
	\end{equation}
	Summing up inequalities (\ref{eq:rn_sum_c}) and  (\ref{eq:rn_sum_a}), we obtain
	$$  3\gamma_{sR}(R_n) = 3 \left(  \sum_{i=0}^{n-1} \overline{f}(a_i) +  \sum_{i=0}^{n-1} \overline{f}(b_i) +  \sum_{i=0}^{n-1} \overline{f}(c_i) \right) \geqslant 2n$$ from where we have 
	
	$$ \gamma_{sR}(R_n) \geqslant \frac{2 n}{3},$$ for any $n$. 
	Now, we consider three cases:
	
	Case 1:  $n = 3k$ 
	
	$\gamma_{sR}(R_{3k})\geqslant \frac{6k}{3} =2k$, which implies that the lower bound is equal to the upper bound, i.e. $\gamma_{sR}(R_{3k}) = 2k$.
	
	Case 2: $n = 3k+1$
	
	$\gamma_{sR}(R_{3k+1})\geqslant \frac{6k+2}{3} =2k + \frac23$. Since  $\gamma_{sR}(R_{3k+1})$ must be an integer, we have that   $\gamma_{sR}(R_{3k+1})\geqslant 2k+1$. In this case, we obtained that the lower bound is less than the upper bound by one, so  $2k+1\leqslant \gamma_{sR}(R_{3k+1})\leqslant 2k+2$.
	
	Case 3: $n = 3k+2$
	
	$\gamma_{sR}(R_{3k+2})\geqslant \frac{6k+4}{3} =2k + 1 + \frac13$. Again, since   $\gamma_{sR}(R_{3k+2})$ must be an integer, we have that  $\gamma_{sR}(R_{3k+2})=2k+2$. In this case, we proved that the lower bound is equal to the upper one, so $\gamma_{sR}(R_{3k+2}) = 2k+2$. 
	This concludes our proof. 
\end{proof}

\begin{rem}
Similiraly as for $A_n$, the re-constructioned function obtained by replacing all $0$ labels by $-1$ in the (optimal) RDF for the basic RDP from ~\cite{kartelj2021roman} is not an admissible  SRDF for $R_n$. The reason again lies in  condition (\ref{eq:c2srdp}) which is not satisfied. Therefore, the (optimal) solution of RDP on graph $R_n$ cannot be directly used in proving the above theorem. 
\end{rem}

\begin{rem}	
	Please note that the lower bound in (\ref{eq:lower_srdp2}) gives $\gamma_{sR}(R_n)\geqslant-\frac{3}{16}n$.  Thus, this bound is cannot be directly used for the proof of 
	 Theorem~\ref{thm:rn_exact}.
\end{rem} 
	
	\subsection{Signed total Roman domination number for $S_n$}
	This class of planar graphs is introduced in~\cite{imran2013classes}. 
	It is defined as a graph $S_n=(V(S_n), E(S_n))$ in the following way: 
	$$V(S_n) = \{a_i, b_i, c_i, d_i \mid i=0,\ldots, n-1\}$$ and 
	$$E(S_n) = \{a_i a_{i+1}, b_i b_{i+1,} c_i c_{i+1}, d_i d_{i+1}, a_i b_i, b_i c_i, c_i d_i, b_{i+1}c_i \mid i=0,\ldots, n-1 \}.$$ One such graph is displayed in Figure~\ref{fig:sn}. 
	\begin{figure}[htbp]
	\centering\setlength\unitlength{1mm}

\setlength\unitlength{1mm}
\begin{picture}(130,60)
\thicklines
\tiny
\put(47.7,5.0){\circle*{1}} \put(47.7,5.0){\line(1,1){7.3}} \put(47.7,5.0){\line(-3,-1){8.8}} \put(47.7,5.0){\line(-1,5){1.8}} 
\put(55.0,12.3){\circle*{1}} \put(55.0,12.3){\line(4,1){10.0}} \put(55.0,12.3){\line(-5,1){9.1}} \put(55.0,12.3){\line(1,3){3.0}} 
\put(65.0,15.0){\circle*{1}} \put(65.0,15.0){\line(4,-1){10.0}} \put(65.0,15.0){\line(-5,4){7.0}} \put(65.0,15.0){\line(5,4){7.0}} 
\put(75.0,12.3){\circle*{1}} \put(75.0,12.3){\line(1,-1){7.3}} \put(75.0,12.3){\line(-1,3){3.0}} \put(75.0,12.3){\line(5,1){9.1}} 
\put(82.3,5.0){\circle*{1}} \put(82.3,5.0){\line(1,5){1.8}} \put(82.3,5.0){\line(3,-1){8.8}} 
\put(38.9,2.0){\circle*{1}} \put(38.9,2.0){\line(3,5){7.0}} \put(38.9,2.0){\line(-4,1){9.7}} 
\put(45.9,14.1){\circle*{1}} \put(45.9,14.1){\line(5,3){12.1}} \put(45.9,14.1){\line(-1,1){7.1}} 
\put(58.0,21.1){\circle*{1}} \put(58.0,21.1){\line(1,0){14.0}} \put(58.0,21.1){\line(-1,4){2.6}} 
\put(72.0,21.1){\circle*{1}} \put(72.0,21.1){\line(5,-3){12.1}} \put(72.0,21.1){\line(1,4){2.6}} 
\put(84.1,14.1){\circle*{1}} \put(84.1,14.1){\line(3,-5){7.0}} \put(84.1,14.1){\line(1,1){7.1}} 
\put(91.1,2.0){\circle*{1}} \put(91.1,2.0){\line(4,1){9.7}} 
\put(29.3,4.6){\circle*{1}} \put(29.3,4.6){\line(3,5){9.6}} \put(29.3,4.6){\line(-4,1){9.7}} 
\put(38.8,21.2){\circle*{1}} \put(38.8,21.2){\line(5,3){16.6}} \put(38.8,21.2){\line(-1,1){7.1}} 
\put(55.4,30.7){\circle*{1}} \put(55.4,30.7){\line(1,0){19.2}} \put(55.4,30.7){\line(-1,4){2.6}} 
\put(74.6,30.7){\circle*{1}} \put(74.6,30.7){\line(5,-3){16.6}} \put(74.6,30.7){\line(1,4){2.6}} 
\put(91.2,21.2){\circle*{1}} \put(91.2,21.2){\line(3,-5){9.6}} \put(91.2,21.2){\line(1,1){7.1}} 
\put(100.7,4.6){\circle*{1}} \put(100.7,4.6){\line(4,1){9.7}} 
\put(19.6,7.2){\circle*{1}} \put(19.6,7.2){\line(3,5){12.2}} 
\put(31.8,28.2){\circle*{1}} \put(31.8,28.2){\line(5,3){21.1}} 
\put(52.8,40.4){\circle*{1}} \put(52.8,40.4){\line(1,0){24.3}} 
\put(77.2,40.4){\circle*{1}} \put(77.2,40.4){\line(5,-3){21.1}} 
\put(98.2,28.2){\circle*{1}} \put(98.2,28.2){\line(3,-5){12.2}} 
\put(110.4,7.2){\circle*{1}} 
\thinlines
\put(47.7,5.0){\line(-1,-2){2}} \put(82.3,5.0){\line(1,-2){2}}
\put(38.9,2.0){\line(1,-1){4}} \put(91.1,2.0){\line(-1,-1){4}}
\put(38.9,2.0){\line(0,-1){5}} \put(91.1,2.0){\line(0,-1){5}}
\put(29.3,4.6){\line(0,-1){5}} \put(100.7,4.6){\line(0,-1){5}}
\put(19.6,7.2){\line(0,-1){5}} \put(110.4,7.2){\line(0,-1){5}}

\put(48.3,3.5){$a_2$} \put(54.5,9.7){$a_1$} \put(63.0,12.0){$a_0$} \put(71.5,9.7){$a_{n-1}$} \put(77.7,3.5){$a_{n-2}$} 
\put(37.6,4.3){$b_2$} \put(45.8,16.3){$b_1$} \put(58.9,22.5){$b_0$} \put(73.3,21.4){$b_{n-1}$} \put(85.3,13.2){$b_{n-2}$} \put(91.5,0.1){$b_{n-3}$} 
\put(27.9,6.7){$c_2$} \put(38.6,23.2){$c_1$} \put(56.1,32.2){$c_0$} \put(75.7,31.1){$c_{n-1}$} \put(92.2,20.4){$c_{n-2}$} \put(101.2,2.9){$c_{n-3}$} 
\put(16.7,7.7){$d_2$} \put(29.4,29.6){$d_1$} \put(51.3,42.3){$d_0$} \put(76.7,42.3){$d_{n-1}$} \put(98.6,29.6){$d_{n-2}$} \put(111.3,7.7){$d_{n-3}$} 

\end{picture}
\caption{\label{fig:sn} \small The graph of convex polytope $S_n$}
\end{figure}
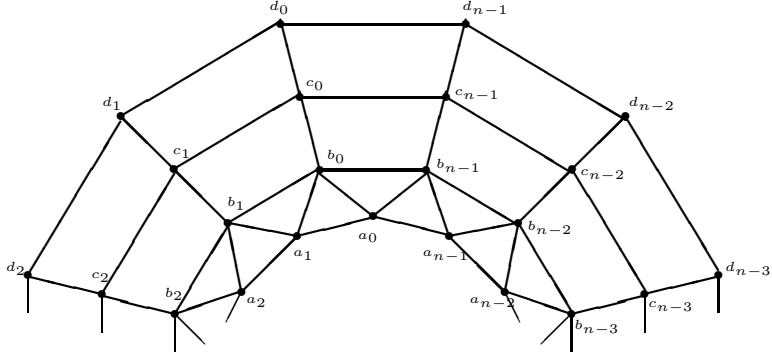
	
	Polytope $S_n$ has four different classes of vertices, named $a$, $b$, $c$ and $d$. This polytope can be obtained from polytope $R_n$ adding new class of vertices $d$ on the following way. Every $d_i$ is adjecent with $c_i$, $d_{i-1}$ and $d_{i+1}$.  We used information about adjacent vertices in the proof of Theorem 3.  
	
	\begin{thm}
		For $n\geq 5$, 	$\gamma_{stR}(S_n) = n$.

	\end{thm}
\begin{proof}	
	{\em Step 1. \underline{Upper bound}}
	
	Let us define a STRD function $f:V \mapsto \{-1,1,2\}$ which partition  the set $V(S_n)$ into  $(V_{-1},V_1,V_2)$ as follows. Let $V_2 = \{b_{i} \mid i=0, \ldots ,n-1\}$, $V_1 = \{d_{i} \mid i=0, \ldots ,n-1\}$ and $V_0=\{a_i,c_i \mid i=0, \ldots ,n-1\}$. 
	Then, it holds $$f(V(S_n)) = 2 \cdot |V_2| +|V_{1}|+ (-1) \cdot |V_{-1}| = 2n+n-2n = n.$$   Let us prove that such $f$ is an STRD function.
	
	From the definition of function $f$, any vertex from $V_{-1}$ is either a $a$--vertex or a $c$--vertex. In grapg $S_n$, each $a$--vertex is adjacent with exactly two $b$--vertices and each $c$--vertex is adjacent with exactly one $b$--vertex. Since every $b$--vertex is labeled by 2, it holds that each vertex from  $V_{-1}$ has its neighbor from $V_2$.
	
	Now we prove that, for each vertex, the condition (\ref{eq:c2strdp}) is satisfied.
	
	Let $a_i$ be an arbitrary $a$--vertex.
	$$s(a_i)^{tot} = f(a_{i-1})+f(a_{i+1})+f(b_{i-1})+f(b_i) = -1-1+2+2 = 2$$
	
	For an arbitrary vertex $b_i$, it holds 
	
	\begin{align*}
		s(b_i)^{tot} &= f(a_i)+f(a_{i+1})+f(b_{i-1})+f(b_{i+1})+f(c_i) \\
		&=-1-1+2+2-1=1
	\end{align*}

	For any vertex $c_i$, it holds
	
	$$s(c_i)^{tot} = f(c_{i-1})+f(c_{i+1})+f(b_{i})+f(d_{i}) = -1-1+2+1 = 1$$
	
	For an arbitrary vertex $d_i$, it holds 
	
	$$s(d_i)^{tot} = f(c_{i}) + f(d_{i-1})+f(d_{i+1})= -1+1+1= 1$$
	
	From the above facts, it implies that $\gamma_{stR}(S_n) \leqslant n$.
	
	{\em Step 2. \underline{Lower bound}}
	
	It remains to prove that $\gamma_{stR}(S_n) \geqslant n$. Let $\overline{f}$ be an arbitrary STRD function. We will prove that $	\overline{f}(V(S_n)) \geqslant n$.
	
	For each $b_i$  from the structure of graph $S_n$ we get 
	
	$$ \overline{f}(a_{i})+\overline{f}(a_{i+1})+\overline{f}(b_{i-1})+\overline{f}(b_i+1) +\overline{f}(c_{i})  \geqslant 1$$
	
	Summing up these inequalities for all $i=0,\ldots, n-1$, we obtain
	
	\begin{equation} 2\sum_{i=0}^{n-1}\overline{f}(a_i)+2 \sum_{i=0}^{n-1}\overline{f}(b_i) + \sum_{i=0}^{n-1}\overline{f}(c_i)\geqslant n
		\label{eqn:Sb}\end{equation}
	
	For each $d_i$, from the structure of graph $S_n$ we get 
	
	$$ \overline{f}(c_{i})+\overline{f}(d_{i-1})+\overline{f}(d_{i+1}) \geqslant 1$$
	
	Summing up these inequalities for each $i=0,\ldots, n-1$, we get
	
	\begin{equation} \sum_{i=0}^{n-1}\overline{f}(c_i)+2 \sum_{i=0}^{n-1}\overline{f}(d_i)\geqslant n
		\label{eqn:Sd}\end{equation}
	
	Finally, by  summing up inequalities (\ref{eqn:Sb}) and (\ref{eqn:Sd}) we obtain
	
	\begin{equation}
		2\overline{f}(V(S_n)) \geqslant 2n
	\end{equation}

	which concludes the proof than $\gamma_{stR}(S_n) \geqslant n$, 
	which proves the theorem. 
\end{proof}	
\begin{rem}	
 The lower bound in (\ref{eq:lower_strdp1}) gives $\gamma_{stR}(S_n)\geqslant\left\lceil-\frac{4}{13}n \right\rceil$ and  the upper bound in (\ref{eq:upper_strdp1}) gives $\gamma_{stR}(S_n)\leqslant 4n-1$.  Hence, these bounds could not be directly incorporated in proving the above theorem.
\end{rem}
	
	\subsection{Signed total Roman domination number for $T_n$}~\label{sec:t_n}
	This class of  graphs  is introduced in~\cite{imran2010families}. 
	It is defined as $T_n=(V(T_n), E(T_n))$,  where 
	$$V(T_n) = \{a_i, b_i, c_i ,d_i\mid i=0,\ldots, n-1\}$$ represents  vertices and 
	$$E(T_n) = \{a_i a_{i+1}, b_i b_{i+1,} c_i c_{i+1},d_i d_{i+1}, a_i b_i, b_i c_i, c_i d_i, a_{i+1}b_i, c_{i} d_{i+1} \mid i=0,\ldots, n-1 \}$$ represents edges of the graph.  One of these graphs is displayed in  Figure~\ref{fig:tn}.
	\begin{figure}[htbp]
	\centering\setlength\unitlength{1mm}

\setlength\unitlength{1mm}
\begin{picture}(130,60)
\thicklines
\tiny
\put(47.7,5.0){\circle*{1}} \put(47.7,5.0){\line(1,1){7.3}} \put(47.7,5.0){\line(-4,-1){10.7}} \put(47.7,5.0){\line(-1,3){3.2}} 
\put(55.0,12.3){\circle*{1}} \put(55.0,12.3){\line(4,1){10.0}} \put(55.0,12.3){\line(-3,1){10.5}} \put(55.0,12.3){\line(1,4){2.5}} 
\put(65.0,15.0){\circle*{1}} \put(65.0,15.0){\line(4,-1){10.0}} \put(65.0,15.0){\line(-1,1){7.5}} \put(65.0,15.0){\line(1,1){7.5}} 
\put(75.0,12.3){\circle*{1}} \put(75.0,12.3){\line(1,-1){7.3}} \put(75.0,12.3){\line(-1,4){2.5}} \put(75.0,12.3){\line(3,1){10.5}} 
\put(82.3,5.0){\circle*{1}} \put(82.3,5.0){\line(1,3){3.2}} \put(82.3,5.0){\line(4,-1){10.7}} 
\put(37.0,2.5){\circle*{1}} \put(37.0,2.5){\line(3,5){7.5}} \put(37.0,2.5){\line(-4,1){8.7}} 
\put(44.5,15.5){\circle*{1}} \put(44.5,15.5){\line(5,3){13.0}} \put(44.5,15.5){\line(-1,1){6.4}} 
\put(57.5,23.0){\circle*{1}} \put(57.5,23.0){\line(1,0){15.0}} \put(57.5,23.0){\line(-1,4){2.3}} 
\put(72.5,23.0){\circle*{1}} \put(72.5,23.0){\line(5,-3){13.0}} \put(72.5,23.0){\line(1,4){2.3}} 
\put(85.5,15.5){\circle*{1}} \put(85.5,15.5){\line(3,-5){7.5}} \put(85.5,15.5){\line(1,1){6.4}} 
\put(93.0,2.5){\circle*{1}} \put(93.0,2.5){\line(4,1){8.7}} 
\put(28.3,4.8){\circle*{1}} \put(28.3,4.8){\line(3,5){9.8}} \put(38.1,21.9){\line(-5,-1){14.7}} \put(28.3,4.8){\line(-1,3){4.9}} 
\put(38.1,21.9){\circle*{1}} \put(38.1,21.9){\line(5,3){17.0}} \put(55.2,31.7){\line(-3,1){14.2}} \put(38.1,21.9){\line(1,5){2.9}} 
\put(55.2,31.7){\circle*{1}} \put(55.2,31.7){\line(1,0){19.7}} \put(74.8,31.7){\line(-5,6){9.8}} \put(55.2,31.7){\line(5,6){9.8}} 
\put(74.8,31.7){\circle*{1}} \put(74.8,31.7){\line(5,-3){17.0}} \put(91.9,21.9){\line(-1,5){2.9}} \put(74.8,31.7){\line(3,1){14.2}} 
\put(91.9,21.9){\circle*{1}} \put(91.9,21.9){\line(3,-5){9.8}} \put(101.7,4.8){\line(1,3){4.9}} \put(91.9,21.9){\line(5,-1){14.7}} 
\put(101.7,4.8){\circle*{1}} 
\put(23.4,19.0){\circle*{1}} \put(23.4,19.0){\line(1,1){17.6}} 
\put(41.0,36.6){\circle*{1}} \put(41.0,36.6){\line(4,1){24.0}} 
\put(65.0,43.0){\circle*{1}} \put(65.0,43.0){\line(4,-1){24.0}} 
\put(89.0,36.6){\circle*{1}} \put(89.0,36.6){\line(1,-1){17.6}} 
\put(106.6,19.0){\circle*{1}}
\thinlines
\put(47.7,5.0){\line(-1,-2){2}} \put(82.3,5.0){\line(1,-2){2}}
\put(37.0,2.5){\line(1,-1){4}} \put(93.0,2.5){\line(-1,-1){4}}
\put(37.0,2.5){\line(0,-1){5}} \put(93.0,2.5){\line(0,-1){5}}
\put(28.3,4.8){\line(0,-1){5}} \put(101.7,4.8){\line(0,-1){5}}
\put(28.3,4.8){\line(-1,-1){4}} \put(101.7,4.8){\line(1,-1){4}}
\put(23.4,19.0){\line(-1,-3){2}} \put(106.6,19.0){\line(1,-3){2}}

\put(48.3,3.5){$a_2$} \put(54.5,9.7){$a_1$} \put(63.0,12.0){$a_0$} \put(71.5,9.7){$a_{n-1}$} \put(77.7,3.5){$a_{n-2}$} 
\put(35.7,5.0){$b_2$} \put(44.5,17.8){$b_1$} \put(58.5,24.5){$b_0$} \put(74.0,23.3){$b_{n-1}$} \put(86.8,14.5){$b_{n-2}$} \put(93.5,0.5){$b_{n-3}$} 
\put(25.4,5.4){$c_2$} \put(35.7,23.3){$c_1$} \put(53.6,33.6){$c_0$} \put(74.4,33.6){$c_{n-1}$} \put(92.3,23.3){$c_{n-2}$} \put(102.6,5.4){$c_{n-3}$} 
\put(19.7,20.0){$d_2$} \put(38.0,38.3){$d_1$} \put(63.0,45.0){$d_0$} \put(88.0,38.3){$d_{n-1}$} \put(106.3,20.0){$d_{n-2}$}

\end{picture}
\caption{\label{fig:tn} \small The graph of convex polytope $T_n$}
\end{figure}
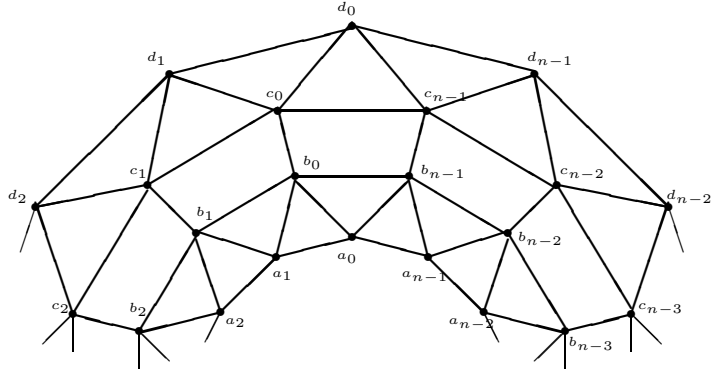
	
	Polytope $T_n$ also can be obtained from polytope $R_n$ by adding a new class of vertices $d$. Every $d_i$ is adjecent with $c_{i-1}$, $c_i$, $d_{i-1}$ and $d_{i+1}$. We used information about adjacent vertices in the proof of Theorem 4.  
	\begin{thm}\label{thm:strdTn}
		$\gamma_{stR}(T_{2k})=2k, k\geqslant 3$ and $2k+1\leqslant \gamma_{stR}(T_{2k+1})\leqslant 2k+2$, $k \geqslant 2$.
	\end{thm}	
	\begin{proof}
		Step 1. \emph{\textit{Upper bound}.} 
		
		\rm Case 1: $n = 2k$.\\
		In this case we define function $f$ by partitioning the set $V(T_{2k})$ as follows:
		
		$V_2 = \{b_{2i} \mid i=0, \ldots ,k-1\}\cup\{c_{2i+1}\mid i =0,\ldots, k-1\}$,    
		$V_1 = \{b_{2i+1} \mid i=0, \ldots ,k-1\}\cup\{c_{2i}\mid i = 0,\ldots, k-1\}$ and  $V_{-1} = \{a_{i} \mid i=0,\ldots,2k-1\}\cup\{d_{i}\mid i=0, \ldots,2k-1\}.$ It holds
		
		$$f(V(T_{2k})) = 2|V_2| +|V_1|+ (-1) \cdot |V_{-1}| = 2\cdot2k+2k-4k = 2k.$$ 
		Let us prove that $f$ is an STRD function. \\
		From the structure of the polytope $T_n$ and the constructed function $f$, we get that every $a$--vertex is adjacent to one $b$--vertex which belongs to the set $V_2$ and every $d$--vertex is adjacent to one $c$--vertex from $V_2$. Since the set $V_{-1}$ is consisted only from $a$ and $d$--vertices,   the condition (\ref{eq:c1}) is proved.
		For any  vertex $a_i$, it holds
		$$s(a_{2i})^{tot} = f(a_{2i-1})+f(a_{2i+1})+f(b_{2i-1})+f(b_{2i}) =-1-1+ 1+2=1$$
		and
		$$s(a_{2i+1})^{tot}=f(a_{2i})+f(a_{2i+2})+f(b_{2i})+f(b_{2i+1}) =-1-1+ 2+1=1.$$
		
		Similarly,  for $d$--vertices it holds
		$$s(d_{2i})^{tot} = f(c_{2i-1})+f(c_{2i})+f(d_{2i-1})+f(d_{2i}) =1+2-1-1=1$$
		and
		$$s(d_{2i+1})^{tot}=f(c_{2i})+f(c_{2i+1})+f(d_{2i})+f(b_{2i+2}) =2+1-1-1=1.$$
		Further on, for  $b$--vertices we obtain
		$$s(b_{2i})^{tot} = f(a_{2i})+f(a_{2i+1})+f(b_{2i-1})+f(b_{2i+1})+f(c_{2i}) =-1-1+1+1+1=1$$
		and
		$$s(b_{2i+1})^{tot} = f(a_{2i+1})+f(a_{2i+2})+f(b_{2i})+f(b_{2i+2})+f(c_{2i+1}) =-1-1+2+2+2=4,$$
		whereas for $c$--vertices we have
		$$s(c_{2i})^{tot} = f(b_{2i})+f(c_{2i-1})+f(c_{2i+1})+f(d_{2i-1})+f(d_{2i}) =2+2+2-1-1=4$$
		and
		$$s(c_{2i+1})^{tot} = f(b_{2i+1})+f(c_{2i})+f(c_{2i+2})+f(d_{2i})+f(d_{2i+1}) =1+1+1-1-1=1.$$
		Therefore, function $f$ satisfies condition  (\ref{eq:c2strdp})\\
		Case 2: $n = 2k+1$. \\
		Here we define  function $f$ as follows: each vertex $v\in V$ whose index is some $i =0,\ldots,2k-1$ is labeled  by the value as in the Case 1, vertices $a_{2k}$ and $d_{2k}$ are labeled by -1, vertex $b_{2k}$ is labeled by 1 and vertex $c_{2k}$ is labeled by 2. For such constructed function $f$ we have 
		$$f(V(T_{2k+1})) = 2|V_2| +|V_1|+ (-1) \cdot |V_{-1}| = 2(2k+2)+2k-(4k+2) = 2k+2.$$ 
		For proving that $f$ is an STRD function we do it analogously as for Case 1.  

		Step 2. \emph{\textit{Lower bound}.} \\
		Let $\overline{f}$ be an arbitrary STRD function.\\
		From condition (\ref{eq:c2strdp}) it follows
		$$ s(a_i)^{tot} = \overline{f}(a_{i-1})+\overline{f}(a_{i+1})+\overline{f}(b_{i-1})+\overline{f}(b_i)  \geqslant 1$$
		If we sum up  these  inequalities for each $i=0,\ldots, n-1$, we get
		
		\begin{equation} 2\sum_{i=0}^{n-1}\overline{f}(a_i)+2 \sum_{i=0}^{n-1}\overline{f}(b_i) \geqslant n
			\label{eqn:Ta}\end{equation}
		
		The same condition gives
		
		$$s(d_i)^{tot} =  \overline{f}(c_{i-1})+\overline{f}(c_{i})+\overline{f}(d_{i-1})+\overline{f}(d_{i+1}) \geqslant 1$$
		
		By summing up these inequalities for each $i=0,\ldots, n-1$, we obtain
		
		\begin{equation} 2\sum_{i=0}^{n-1}\overline{f}(c_i)+2 \sum_{i=0}^{n-1}\overline{f}(d_i)\geqslant n
			\label{eqn:Td}\end{equation}
		
		Now, summing up inequalities (\ref{eqn:Ta}) and (\ref{eqn:Td}), we obtain

		\begin{equation}
			2\overline{f}(V(T_n)) \geqslant 2n
		\end{equation}
		Therefore, we showed that $\gamma_{stR}(T_n) \geqslant n$, which concludes the proof.
\end{proof}

\begin{rem}
We note that for graph $T_n$,  the re-constructioned function obtained by replacing all $0$ labels by $-1$ in the (optimal) RDF for the basic RDP   from ~\cite{kartelj2021roman} is not an admissible  STRDF. The reason again lies in  condition (\ref{eq:c2srdp}) which is not satisfied in that case. Therefore, the solution of RDP on graph $T_n$ cannot be directly used in proving the above theorem.
\end{rem}

\begin{rem}	
	From the theoretical lower bound  (\ref{eq:lower_strdp1}) we get $\gamma_{stR}(T_n)\geqslant\lceil\frac{2}{7}n\rceil$ and for the theoretical upper bound  (\ref{eq:upper_strdp1}) we get $\gamma_{stR}(T_n)\leqslant 4n-1$. It is clearly seen that these bounds could not be directly used in proving the above theorem. 
\end{rem}		
	\section{About some theoretical lower and upper bounds of some polytopes }
	
	
	In this section we give some theoretical lower and  upper bounds on $\gamma_{sR}$ and $\gamma_{stR}$ for some other classes of convex polytopes.
	
	\subsection{Signed Roman domination number for $T_n$}
	Definition of graph $T_n$ is given in Section~\ref{sec:t_n}. The following theorem holds.
	\begin{thm}
		$\gamma_{sR}(T_n) \in [\frac{3n}{4}, n]$, for all $n \geqslant 5$.
	\end{thm}
\begin{proof} 
	Let us prove that $\gamma_{sR}(T_n) \leqslant  n$. In order to do that, we construct a function $f$ on graph $T_n$ which is an SRD function. Each $a_i$ we label by $1$, then each $b_i$ and $d_i$ we label by $-1$, and each $c_i$ by 2, $i =0,\ldots,n-1$. In this way, we indeed obtain an SRDP: from the structure of graph $T_n$, each $b$--vertex, which is labeled by -1, is connected to (exactly) one $c$--vertex, labeled by 2; also,  each $d$--vertex, also labeled by -1,  is connected to exactly one $c$-vertex (labeled by 2). This implies that condition (\ref{eq:c1}) is also fulfilled. Since
	
	$$s(a_i) = f(a_i) + f(a_{i-1}) + f(a_{i+1}) + f(b_i) + f(b_{i-1}) = 1$$
	and 
	$$s(b_i) = f(b_i) + f(b_{i-1}) + f(b_{i+1}) + f(a_i) + f(b_{i+1}) + f(c_i) = -3 + 2 +2 = 1$$
	and 
	$$s(c_i) = f(c_{i-1}) + f(c_i) + f(c_{i+1}) + f(b_i) + f(d_{i-1}) + f(d_{i}) = 
	6 -2-1  = 3$$
	and 
	$$s(d_i) = f(d_i) + f(d_{i-1}) + f(d_{i+1}) + f(c_i) + f(c_{i+1}) = -3 + 4 = 1$$
	
	holds, we have that condition (\ref{eq:c2srdp} is fulfilled. 
	
	Let us prove now that $\gamma_{sR}(T_n) \geqslant \frac{3n}{4}$. First, we introduce some notation which is used in our proofs in this section. Let $A_{i}$ be the number of $a$--vertices that are assigned by $i \in \{-1, 1, 2\}$. Similarly, let $B_{i}$, $C_{i}$ and $D_i$ be 
	the number of $b$--vertices, $c$--vertices and $d$--vertices assigned by $i \in \{-1,1,2\}$, respectively.  
	Further,  from the structure of graph $T_n$ and condition (\ref{eq:c2srdp}), by summing up the inequalities grouped by different types of vertices ($a$--vertices, $b$--vertices, etc.), we get the following:
	\begin{align}
		& 3(-A_{-1} + A_1 + 2 A_{2}) + 2 ( -B_{-1} + B_{1} + 2 B_2 ) \geqslant n \label{eq:tn1}\\
		& 2 (-A_{-1} + A_1 + 2 A_2 ) + 3( - B_{-1} + B_1 + 2 B_2 ) + (-C_{-1}+ C_1 + 2 C_2) \geqslant n \label{eq:tn2}\\
		& (-B_{-1} +B_1 + 2 B_2) + 3(-C_{-1} + C_1 + 2 C_2 ) + 2 (- D_{-1} + D_1 + 2 D_2) \geqslant n \label{eq:tn3} \\
		& 2 (-C_{-1} + C_1 + 2 C_2 ) + 3 (- D_{-1} + D_1 + 2 D_2) \geqslant n. \label{eq:tn4}
	\end{align}
	Note that for an arbitraty SRD function $\overline{f}$ it holds\\
	$$\overline{f}(V(T_n)) = \sum_{i\in \{-1,1,2\}} i ( A_{i}+B_i+C_i+D_i) $$
	Now, multiplying inequalities
	\begin{itemize}
		\item  $(\ref{eq:tn1})$ by $\frac{1}{4}$
		\item $(\ref{eq:tn2})$ by $\frac{1}{8}$
		\item  $(\ref{eq:tn3})$ by $\frac{1}{8}$
		\item  $(\ref{eq:tn4})$ by $\frac{1}{4}$
	\end{itemize} 
	
	and summing up all of them,  we obtain 
	
	$$\overline{f}(V(T_n)) \geqslant  \frac{1}{4} n + \frac{1}{8} n + \frac{1}{8} n + \frac{1}{4}n = \frac{3n}{4},$$ which concludes the proof.  
\end{proof}

\begin{rem}
	We applied total enumeration for $T_5$  and obtained $\gamma_{sR}(T_5)=5$ which represents the upper bound.
\end{rem}

\begin{rem}	
 
We emphasize that from the inequality (\ref{eq:lower_srdp2})  we have $\gamma_{sR}(T_n)\geqslant\frac{4}{17}n$. Therefore, this lower bound could not be utilized in  proving 
the above theorem. \\
\end{rem}

	\subsection{Signed Roman domination number for $Q_n$}
	
	This class of graphs is introduced in~\cite{baca1992magic}. Polytope $Q_n$ is defined as $Q_n=(V(Q_n), E(Q_n))$,  where 
	
	$$V(Q_n) = \{a_i, b_i, c_i ,d_i\mid i=0,\ldots, n-1\}$$ represents the vertices and 
	
	$$E(Q_n) = \{a_i a_{i+1}, b_i b_{i+1,} d_i d_{i+1},a_i b_{i}, b_i c_i, c_i d_i, b_{i+1} c_i \mid i=0,\ldots, n-1 \}$$ represents the edges of this graph.  One of these graphs is displayed in  Figure~\ref{fig:qn}. 
	
	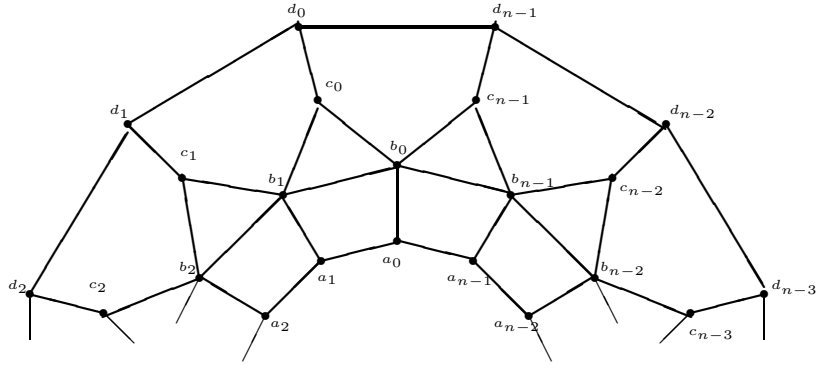
\begin{figure}[htbp]
	\centering\setlength\unitlength{1mm}

	\setlength\unitlength{1mm}

\setlength\unitlength{1mm}
\begin{picture}(130,60)
\thicklines
\tiny
\put(47.7,5.0){\circle*{1}} \put(47.7,5.0){\line(1,1){7.3}} \put(47.7,5.0){\line(-5,3){8.7}} 
\put(55.0,12.3){\circle*{1}} \put(55.0,12.3){\line(4,1){10.0}} \put(55.0,12.3){\line(-3,5){5.0}} 
\put(65.0,15.0){\circle*{1}} \put(65.0,15.0){\line(4,-1){10.0}} \put(65.0,15.0){\line(0,1){10.0}} 
\put(75.0,12.3){\circle*{1}} \put(75.0,12.3){\line(1,-1){7.3}} \put(75.0,12.3){\line(3,5){5.0}} 
\put(82.3,5.0){\circle*{1}} \put(82.3,5.0){\line(5,3){8.7}} 
\put(39.0,10.0){\circle*{1}} \put(39.0,10.0){\line(1,1){11.0}} \put(39.0,10.0){\line(-5,-2){12.7}} \put(39.0,10.0){\line(-1,6){2.3}} 
\put(50.0,21.0){\circle*{1}} \put(50.0,21.0){\line(4,1){15.0}} \put(50.0,21.0){\line(-6,1){13.3}} \put(50.0,21.0){\line(2,5){4.6}} 
\put(65.0,25.0){\circle*{1}} \put(65.0,25.0){\line(4,-1){15.0}} \put(65.0,25.0){\line(-5,4){10.4}} \put(65.0,25.0){\line(5,4){10.4}} 
\put(80.0,21.0){\circle*{1}} \put(80.0,21.0){\line(1,-1){11.0}} \put(80.0,21.0){\line(-2,5){4.6}} \put(80.0,21.0){\line(6,1){13.3}} 
\put(91.0,10.0){\circle*{1}} \put(91.0,10.0){\line(1,6){2.3}} \put(91.0,10.0){\line(5,-2){12.7}} 
\put(26.4,5.4){\circle*{1}} \put(26.4,5.4){\line(-4,1){9.7}} 
\put(36.7,23.3){\circle*{1}} \put(36.7,23.3){\line(-1,1){7.1}} 
\put(54.6,33.6){\circle*{1}} \put(54.6,33.6){\line(-1,4){2.6}} 
\put(75.4,33.6){\circle*{1}} \put(75.4,33.6){\line(1,4){2.6}} 
\put(93.3,23.3){\circle*{1}} \put(93.3,23.3){\line(1,1){7.1}} 
\put(103.6,5.4){\circle*{1}} \put(103.6,5.4){\line(4,1){9.7}} 
\put(16.7,7.9){\circle*{1}} \put(16.7,7.9){\line(3,5){12.9}} 
\put(29.6,30.4){\circle*{1}} \put(29.6,30.4){\line(5,3){22.4}} 
\put(52.1,43.3){\circle*{1}} \put(52.1,43.3){\line(1,0){25.9}} 
\put(77.9,43.3){\circle*{1}} \put(77.9,43.3){\line(5,-3){22.4}} 
\put(100.4,30.4){\circle*{1}} \put(100.4,30.4){\line(3,-5){12.9}} 
\put(113.3,7.9){\circle*{1}}
\thinlines
\put(47.7,5.0){\line(-1,-2){3}} \put(82.3,5.0){\line(1,-2){3}}
\put(39.0,10.0){\line(-1,-2){3}} \put(91.0,10.0){\line(1,-2){3}}
\put(26.4,5.4){\line(1,-1){4}} \put(103.6,5.4){\line(-1,-1){4}}
\put(16.7,7.9){\line(0,-1){6}} \put(113.3,7.9){\line(0,-1){6}}

\put(48.3,3.5){$a_2$} \put(54.5,9.7){$a_1$} \put(63.0,12.0){$a_0$} \put(71.5,9.7){$a_{n-1}$} \put(77.7,3.5){$a_{n-2}$} 
\put(36.3,11.0){$b_2$} \put(48.0,22.7){$b_1$} \put(64.0,27.0){$b_0$} \put(80.0,22.7){$b_{n-1}$} \put(91.7,11.0){$b_{n-2}$} 
\put(24.4,8.7){$c_2$} \put(36.4,26.2){$c_1$} \put(55.5,35.3){$c_0$} \put(76.7,33.6){$c_{n-1}$} \put(94.2,21.6){$c_{n-2}$} \put(103.3,2.5){$c_{n-3}$} 
\put(13.8,8.5){$d_2$} \put(27.2,31.8){$d_1$} \put(50.5,45.2){$d_0$} \put(77.5,45.2){$d_{n-1}$} \put(100.8,31.8){$d_{n-2}$} \put(114.2,8.5){$d_{n-3}$} 

\end{picture}
\caption{\label{fig:qn} \small The graph of convex polytope $Q_n$}
\end{figure}
	
	In convex polytope $Q_n$ every $a_i$ vertex is adjacent with vertices $a_{i-1}$, $a_{i+1}$ and $b_i$. Every $b_i$ vertex is connected with $a_i$, $b_{i-1}$, $b_{i+1}$, $c_{i-1}$ and $c_i$ vertex. Vertex $c_i$ is  is adjacent with vertices $b_i$, $b_{i+1}$ and $d_i$, while every $d_i$ vertex is connected with $c_i$, $d_{i-1}$ and $d_{i+1}$ vertex. In the proof of the following theorem we used these properties.
	\begin{thm} \label{thm:qn_srdp}
	For $n \geqslant 12$,	$\gamma_{sR}(Q_n) \in [\frac{2n}{3}, n]$. 
	\end{thm}
\begin{proof}
	\em{Step 1. $\gamma_{sR}(Q_n)\leqslant n$}
	
	\rm We define  function $f$ on graph $Q_n$  for which $f(V(Q_n)) = n$ in Table ~\ref{Tab:T1}.
	
	\begin{table}[ht]
		\caption{\rm  Positive values of SRD function on the graph $Q_n$}
		\centering
		\begin{tabular}{ |c| m{4.7 cm}| m{4.4 cm}| }
			\hline
			$n$ & $f(v) = 2$ & $f(v) = 1 $\\
			\hline 
			$3k$ & $ a_{3i},b_{3i+1},d_{3i+2},  i=0,\ldots,k-1$ &$ b_{3i},b_{3i+2},d_{3i},  i=0,\ldots,k-1$ \\  
			\hline
			$3k+1$ & $ a_{3i},b_{3i+1},d_{3i+2}, i=0,\ldots,k-2$   $ a_{3k-3},b_{3k-1},d_{3k-3},d_{3k}$ &$ b_{3i},b_{3i+2},d_{3i},  i=0,\ldots,k-2$ $a_{3k-1},b_{3k-3},c_{3k-2},c_{3k-1}$ \\  
			\hline
			$3k+2$ & $ a_{3i},b_{3i+1},d_{3i+2},  i=0,\ldots,k-4$ $ a_{3k-9},a_{3k-5},a_{3k-2}$ $b_{3k-7},b_{3k-4},b_{3k}$ $d_{3k-9},d_{3k-6},d_{3k-3},d_{3k-2},d_{3k+1}$  &$ b_{3i},b_{3i+2},d_{3i},  i=0,\ldots,k-4$ $a_{3k-7},a_{3k} \ \ \ \ \ \ $
			$b_{3k-9},b_{3k-5},b_{3k-3},b_{3k-2}$ $c_{3k-8},c_{3k-7},c_{3k-1},c_{3k},d_{3k-5}$\\
			\hline
		\end{tabular}
		\label{Tab:T1}
	\end{table}
	
	Let us show now that $f$ is SRD function.
	
	Case 1.  $n = 3k$
	
	Let $v\in V_{-1}$. 
	
	We have
	
	\begin{equation} V_2\cap N(v) = \begin {cases}\{a_{3i},b_{3i+1}\},v = a_{3i+1}\\ 
		\{a_{3i+3}\},v = a_{3i+2}\\ \{b_{3i+1}\},v = c_{3i}\  \mathrm{or} \ v = c_{3i+1}\\ \{d_{3i+2}\}, v = c_{3i+2}\  \mathrm{or}\  v = d_{3i+1}  
	\end{cases}
	\label{eqn:Tncases}
\end{equation}
for $ i=0\ldots, k-1$.

Thus, condition (\ref {eq:c1}) for $n=3k$ holds.

Let now prove the condition (\ref{eq:c2srdp}) is fulfilled.

$$s(a_i) = f(a_{i-1})+f(a_i)+f(a_{i+1})+f(b_i)=\begin{cases}  2, \ i = 3j+1\\
	1, \ \mathrm{otherwise}
\end{cases}$$
$$s(b_i) = f(a_{i})+f(b_{i-1})+f(b_i)+f(b_{i+1})+f(c_{i-1})+f(c_i)
	={\begin{cases} 4, \ i = 3j\\
			1, \ \mathrm{otherwise}
	\end{cases}}$$
$$s(c_i) = f(b_i)+f(b_{i+1})+f(c_{i})+f(d_i)=
\begin{cases} 1, \ i = 3j+1\\
	3,\ \mathrm{otherwise}
\end{cases}$$
$$s(d_i) =f(c_i)+ f(d_{i-1})+f(d_i)+f(d_{i+1})=1$$

We conclude that condition (\ref{eq:c2srdp}) is satisfied for $n = 3k$.

Case 2. $n = 3k+1$,

Let $v\in V_{-1}$.

Similarly as in Case 1, we have that equation (\ref{eqn:Tncases}) holds for $ i=0\ldots, 3k-5$.

The coverage of the rest of vertices from $V_{-1}$ is shown in Table~\ref{Tab:T2}.  
We get that condition (\ref {eq:c1} is also valid for $n = 3k+1$. 
\begin{table}[ht]
	\caption{\rm  SRD coverage for $Q_{3k+1}$}
	\centering
	\begin{tabular}{ |c| m{1.5cm}| m{1.5cm}|m{1.7cm}| m{1.4cm}|}
		\hline
		$n$ & $v\in V_{-1}$ & $ V_{2}\cap N(v)$& $v\in V_{-1}$ & $ V_{2}\cap N(v)$\\
		\hline 
		$3k+1$ & $ a_{3k-4}$  $c_{3k-4}$    $ a_{3k-2}\ \ \ \ \ \ $  $a_{3k}$         
		& $\{a_{3k-3}\}$ $\{d_{3k-4}\}$   $\{a_{3k-3}\}$ $\{a_0\}\ \ \ \ \ \ $
		&  $b_{3k-2},b_{3k}$ $c_{3k-3},d_{3k-2}$ $c_{3k},d_{3k-1}$  
		&  $\{b_{3k-1}\}$ $\{d_{3k-3}\}$ $\{d_{3k}\}$\\
		\hline
	\end{tabular}
	\label{Tab:T2}
\end{table}

Let us prove that  condition (\ref{eq:c2srdp}) is satisfied.

For $i = 0,\ldots,3k-5$, the same calculation as in Case 1 holds and for $i = 3k-4,\ldots,3k$  we get:
$$s(a_i) = f(a_{i-1})+f(a_i)+f(a_{i+1})+f(b_i)= 1$$
$$s(b_i) = f(a_{i})+f(b_{i-1})+f(b_i)+f(b_{i+1})+f(c_{i-1})+f(c_i) = 1$$
$$s(c_i) = f(b_i)+f(b_{i+1})+f(c_{i})+f(d_i)=
\begin{cases} 3, \ i = 3k-4\\
	1, \ \mathrm{otherwise}
\end{cases}$$
$$s(d_i) =f(c_i)+ f(d_{i-1})+f(d_i)+f(d_{i+1})=
\begin{cases}  2, \ i = 3k-4\ \mathrm{or}\  i = 3k-3\\
	1,\ \mathrm{otherwise}
\end{cases}$$
Case 3. $n = 3k+2$

The  equation (\ref{eqn:Tncases}) holds for $i = 0,\ldots,3k-7$.  The coverage of the rest of vertices from $V_{-1}$ is shown in Table \ref{Tab:T3}.
\begin{table}[h]
	\caption{\rm  SRD coverage for  $Q_{3k+2}$}
	\centering
	\begin{tabular}{| c| m{1.7cm}| m{2cm}|m{1.7cm}| m{2cm}|}
		\hline
		$n$& $v\in V_{-1}$ & $ V_{2}\cap N(v)$& $v\in V_{-1}$ & $ V_{2}\cap N(v)$ \\
		\hline 
		$3k+2$&    $a_{3k-6}$ $b_{3k-6}$ $c_{3k-6}$  $ a_{3k-4}$   $c_{3k-5},c_{3k-4}$  $c_{3k-3},d_{3k-4}$
		& $\{a_{3k-5}\}\ \ \ \ \ \ $ $\{b_{3k-7}\}$ $\{d_{3k-6}\}$ $\{a_{3k-5},b_{3k-4}\}$ $\{b_{3k-4}\}\ \ \ \ \ \ $  $\{d_{3k-3}\}$
		&    $ a_{3k-1}\ \ \ \ \ \ $  $a_{3k+1}$ $b_{3k-1},b_{3k+1}$ $c_{3k-2},d_{3k-1}$ $c_{3k+1},d_{3k}$   
		&    $\{a_{3k-2}\}$ $\{a_0\}\ \ \ \ \ \ $ $\{b_{3k}\}$ $\{d_{3k-2}\}$ $\{d_{3k+1}\}$\\
		\hline
	\end{tabular}
	\label{Tab:T3}
\end{table}\\
Therefore, the condition (\ref {eq:c1}) is also satisfied for $n=3k+2$.


Let us prove that  condition (\ref{eq:c2srdp}) is satisfied.

If $i = 0,\ldots,3k-6$, the same calculation as in Case 2 holds.
For $i = 3k-5,\ldots,3k+1$ we obtain: 
$$s(a_i) = f(a_{i-1})+f(a_i)+f(a_{i+1})+f(b_i)=1$$
	$$s(b_i) = f(a_{i})+f(b_{i-1})+f(b_i)+f(b_{i+1})+f(c_{i-1})+f(c_i)\\
	={\begin{cases}  2, \ i = 3k-5\\
			1,\ \mathrm{otherwise}
	\end{cases}}$$
$$s(c_i) = f(b_i)+f(b_{i+1})+f(c_{i})+f(d_i)=
\begin{cases} 3, \ i = 3k-5\ \mathrm {or} \ i = 3k-3\\
	1, \ \mathrm{otherwise}
\end{cases}$$
$$s(d_i) =f(c_i)+ f(d_{i-1})+f(d_i)+f(d_{i+1})=
\begin{cases}   2, \ i = 3k-3\ \mathrm{or}\  i =3k-2\\
	1, \ \mathrm{otherwise}
\end{cases}$$

Step 2.  $\gamma_{sR}(Q_n)\geqslant \frac{2n}3$

It is easy to see that 
\begin{align}
	&A_{-1} + A_{1} + A_{2} = n \label{qn_sum_n_a1} \\
	&B_{-1} + B_{1} + B_{2}  = n \\
	&C_{-1} + C_{1} + C_{2} = n \label{qn_sum_c}\\
	&D_{-1} + D_{1} + D_{2} = n \label{qn_sum_n_d1}.
\end{align}
Further,  exploiting the structure of graph $Q_n$ and condition (\ref{eq:c2srdp}), by summing up the inequalities grouped by different types of vertices ($a$--vertices, $b$--vertices, etc.), we get the following
\begin{align}
	& 3(-A_{-1} + A_1 + 2 A_{2}) +  ( -B_{-1} + B_{1} + 2 B_2 ) \geqslant n \label{eq:qn1}\\
	&  (-A_{-1} + A_1 + 2 A_2 ) + 3( - B_{-1} + B_1 + 2 B_2 ) +2 (-C_{-1}+ C_1 + 2 C_2) \geqslant n \label{eq:qn2}\\
	& (-B_{-1} +B_1 + 2 B_2) + 3(-C_{-1} + C_1 + 2 C_2 ) + 2 (- D_{-1} + D_1 + 2 D_2) \geqslant n \label{eq:qn3} \\
	&  (-C_{-1} + C_1 + 2 C_2 ) + 3 (- D_{-1} + D_1 + 2 D_2) \geqslant n. \label{eq:qn4}
\end{align}
Now, multiplying 
\begin{itemize}
	\item \ref{qn_sum_c} by -$\frac{1}{6}$
	\item \ref{eq:qn1} by $\frac{1}{4}$
	\item \ref{eq:qn2} by $\frac{1}{4}$
	\item \ref{eq:qn4} by $\frac{1}{3}$
\end{itemize} 

and summing up all of them, we get that for any SRD function  $\overline{f}$ 

\begin{align*} 
	\overline{f}(V(Q_n))  \geqslant -\frac{n}{6} +  \frac{n}{4} + \frac{n}{4} + \frac{n}{3} = \frac{2n}{3}
\end{align*} 
holds, which concludes our proof. 
\end{proof}

\begin{rem}
The inequality  (\ref{eq:lower_srdp2}) gives   $\gamma_{sR}(Q_n)\geqslant -\frac{n}{4}$ which cannot be directly used  in proving the above theorem.
\end{rem}

\subsection{Signed Roman domination number for $T_n''$}

These graphs are introduced in~\cite{imran2013classes}. 
It is defined as $T_n''=(V(T_n''), E(T_n''))$,  where 

$$V(T_n'') = \{a_i, b_i, c_i ,d_i\mid i=0,\ldots, n-1\}$$ represents  vertices and

$$E(T_n'') = \{a_i a_{i+1}, b_i b_{i+1}, d_i d_{i+1}, a_i b_i, b_i c_i, c_i d_i, b_{i+1}c_i, a_{i+1} b_i \mid i=0,\ldots, n-1\}.$$

One of these graphs is displayed in  Figure~\ref{fig:t_n_sec}. 

\begin{figure}[htbp]
	\centering\setlength\unitlength{1mm}

\setlength\unitlength{1mm}
\begin{picture}(130,60)
\thicklines
\tiny
\put(47.7,5.0){\circle*{1}} \put(47.7,5.0){\line(1,1){7.3}} \put(47.7,5.0){\line(1,5){2.8}} \put(47.7,5.0){\line(-5,3){7.8}} 
\put(55.0,12.3){\circle*{1}} \put(55.0,12.3){\line(4,1){10.0}} \put(55.0,12.3){\line(5,6){10.0}} \put(55.0,12.3){\line(-3,5){4.5}} 
\put(65.0,15.0){\circle*{1}} \put(65.0,15.0){\line(4,-1){10.0}} \put(65.0,15.0){\line(3,1){14.5}} \put(65.0,15.0){\line(0,1){9.0}} 
\put(75.0,12.3){\circle*{1}} \put(75.0,12.3){\line(1,-1){7.3}} \put(75.0,12.3){\line(5,-1){15.1}} \put(75.0,12.3){\line(3,5){4.5}} 
\put(82.3,5.0){\circle*{1}} \put(82.3,5.0){\line(5,3){7.8}} 
\put(39.9,9.5){\circle*{1}} \put(39.9,9.5){\line(1,1){10.6}} \put(39.9,9.5){\line(-2,-1){9.7}} \put(39.9,9.5){\line(0,1){11.0}} 
\put(50.5,20.1){\circle*{1}} \put(50.5,20.1){\line(4,1){14.5}} \put(50.5,20.1){\line(-1,0){11.0}} \put(50.5,20.1){\line(1,2){5.2}} 
\put(65.0,24.0){\circle*{1}} \put(65.0,24.0){\line(4,-1){14.5}} \put(65.0,24.0){\line(-5,3){9.3}} \put(65.0,24.0){\line(5,3){9.3}} 
\put(79.5,20.1){\circle*{1}} \put(79.5,20.1){\line(1,-1){10.6}} \put(79.5,20.1){\line(-1,2){5.2}} \put(79.5,20.1){\line(1,0){11.0}} 
\put(90.1,9.5){\circle*{1}} \put(90.1,9.5){\line(0,1){11.0}} \put(90.1,9.5){\line(2,-1){9.7}} 
\put(30.2,4.3){\circle*{1}} \put(30.2,4.3){\line(-4,1){9.7}} 
\put(39.5,20.5){\circle*{1}} \put(39.5,20.5){\line(-1,1){7.1}} 
\put(55.7,29.8){\circle*{1}} \put(55.7,29.8){\line(-1,4){2.6}} 
\put(74.3,29.8){\circle*{1}} \put(74.3,29.8){\line(1,4){2.6}} 
\put(90.5,20.5){\circle*{1}} \put(90.5,20.5){\line(1,1){7.1}} 
\put(99.8,4.3){\circle*{1}} \put(99.8,4.3){\line(4,1){9.7}} 
\put(20.6,6.9){\circle*{1}} \put(20.6,6.9){\line(3,5){11.9}} 
\put(32.5,27.5){\circle*{1}} \put(32.5,27.5){\line(5,3){20.6}} 
\put(53.1,39.4){\circle*{1}} \put(53.1,39.4){\line(1,0){23.8}} 
\put(76.9,39.4){\circle*{1}} \put(76.9,39.4){\line(5,-3){20.6}} 
\put(97.5,27.5){\circle*{1}} \put(97.5,27.5){\line(3,-5){11.9}} 
\put(109.4,6.9){\circle*{1}} 
\thinlines
\put(47.7,5.0){\line(-1,-2){3}} \put(82.3,5.0){\line(1,-2){3}}
\put(39.9,9.5){\line(-1,-2){3}} \put(90.1,9.5){\line(1,-2){3}}
\put(30.2,4.3){\line(1,-2){3}} \put(99.8,4.3){\line(-1,-2){3}}
\put(20.6,6.9){\line(0,-1){6}} \put(109.4,6.9){\line(0,-1){6}}

\put(48.3,3.5){$a_2$} \put(54.5,9.7){$a_1$} \put(63.0,12.0){$a_0$} \put(71.5,9.7){$a_{n-1}$} \put(77.7,3.5){$a_{n-2}$} 
\put(37.2,10.5){$b_2$} \put(48.5,21.8){$b_1$} \put(64.0,26.0){$b_0$} \put(79.5,21.8){$b_{n-1}$} \put(90.8,10.5){$b_{n-2}$} 
\put(28.1,7.4){$c_2$} \put(39.0,23.1){$c_1$} \put(56.3,31.4){$c_0$} \put(75.4,29.9){$c_{n-1}$} \put(91.1,19.0){$c_{n-2}$} \put(99.4,1.7){$c_{n-3}$} 
\put(17.6,7.4){$d_2$} \put(30.1,28.9){$d_1$} \put(51.6,41.4){$d_0$} \put(76.4,41.4){$d_{n-1}$} \put(97.9,28.9){$d_{n-2}$} \put(110.4,7.4){$d_{n-3}$}

\end{picture}
\caption{\label{fig:t_n_sec} \small The graph of convex polytope $T_n''$}
\end{figure}
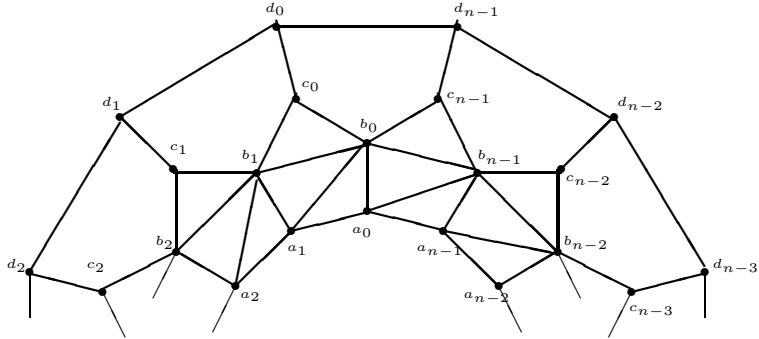

In convex polytope $T_n''$ every $a_i$ vertex is adjacent with vertices $a_{i-1}$, $a_{i+1}$, $b_{i-1}$ and $b_i$.  Every $b_i$ vertex is connected with $a_i$, $a_{i+1}$, $b_{i-1}$, $b_{i+1}$, $c_{i-1}$ and $c_i$ vertex. Vertex $c_i$ is     adjacent with vertices $b_i$, $b_{i+1}$ and $d_i$, whereas every $d_i$ vertex is connected with $c_i$, $d_{i-1}$ and $d_{i+1}$ vertex. In the proof of the following theorem  these properties are used.
\begin{thm}
For $n \geqslant 5$, $\frac{7n}{15}\leqslant\gamma_{sR}(T_{n}'')  \leqslant h(n)$, where $h(n) = \begin{cases} \lceil\frac{2n}{3} \rceil , \ n = 3k\\
 \lceil\frac{2n}{3} \rceil+1 ,\ otherwise\end{cases} $.
\end{thm}
\begin{proof}
\em Step 1. 
	\rm We will prove that function $f$ defined by partitioning 
	$V_2=\{b_i\mid i=0,\ldots,n-1\}\cup\{d_{3i}\mid i=0,\ldots k-1\},\ V_1 = \{d_{3i+1}\mid i = 0,\ldots, k-1\},\ V_{-1} = \{a_i,c_i\mid i=0,\ldots,n-1\}\cup\{d_{3i+2}\mid i=0,\ldots, k-1\}$
	is a SRD function.\\
	First, let calculate the number $f(V(T_n^"))$.
	We get\\
	\begin{align*}f(V(T_n^")) = 2 \cdot |V_2|+|V_1|+(-1) \cdot |V_{-1}|&=\begin{cases}\frac{8n}{3}+\frac{n}{3}-\frac{7n}{3},\ n = 3k,\\
		\frac{8n+4}{3}+\frac{n-1}{3}-\frac{7n-1}{3} ,\ n = 3k+1\\
		\frac{8n+2}{3}+\frac{n+1}{3}-\frac{7n-2}{3},\ n = 3k+2
	\end{cases}\\
& = \begin{cases} \frac{2n}{3},n = 3k,\\
		 \frac{2n+4}{3},n = 3k+1\\
		\frac{2n+5}{3},n = 3k+2
	\end{cases}
\end{align*}
	From the structure of graph $T_n^{"}$, we conclude that each of $a$ and $c$--vertices is adjacent with two $b$--vertices.
	Each vertex $d_{3i+2}$ is adjacent with vertex $d_{3i}$, which is labeled by 2.Thus, condition (\ref{eq:c1}) is satisfied.

	The proof that condition (\ref{eq:c1}) is satisfied for  $a$ and $b$--vertices is the same as in Theorem~\ref{thm:strdTn}.
	
	$$s(c_i) = f(b_i)+f(b_{i+1})+f(c_i)+f(d_i)= \begin{cases} 5, i = 3j,\\
		4,i = 3j+1,\\
		2,i=3j+2. \end{cases}$$
$$s(d_0) =f(c_0)+f(d_0)+f(d_1)+f(d_{n-1})= \begin{cases} 1,\ n =3k\\
4,\ n = 3k+1\\
3,\ n = 3k+2 
\end{cases}$$
	$$s(d_i) = f(c_i)+f(d_{i-1})+f(d_i)+f(d_{i+1}) = 1,\  i = 1,\ldots,n-2$$
$$s(d_{n-1}) =f(c_{n-1})+f(d_0)+f(d_{n-2})+f(d_{n-1}) = \begin{cases} 1,\ n =3k\\
2,\ n = 3k+1\\
4,\ n = 3k+2 
\end{cases}$$

	Therefore, $f$ satisfies condition (\ref{eq:c2srdp}), so $f$ is an SRD function.

	{\em Step 2. } It is easy to see that
	\begin{align}
		&A_{-1} + A_{1} + A_{2} = n \label{sum_tnsec_a} \\
		&B_{-1} + B_{1} + B_{2}  = n \label{sum_tnsec_b}\\
		&C_{-1} + C_{1} + C_{2} = n \label{sum_tnsec_c} \\
		&D_{-1} + D_{1} + D_{2} = n \label{sum_tnsec_d}.
	\end{align}

	Further,  exploiting the structure of graph $T_n''$ and condition (\ref{eq:c2srdp}), by summing up the inequalities grouped by different types of vertices ($a$--vertices, $b$--vertices, etc.), we get the following
	\begin{align}
		& 3(-A_{-1} + A_1 + 2 A_{2}) + 2 ( -B_{-1} + B_{1} + 2 B_2 ) \geqslant n \label{eq:tn"1}\\
		& 2(-A_{-1} + A_1 + 2 A_2 ) + 3( - B_{-1} + B_1 + 2 B_2 ) +2 (-C_{-1}+ C_1 + 2 C_2) \geqslant n \label{eq:tn"2}\\
		& 2(-B_{-1} +B_1 + 2 B_2) + (-C_{-1} + C_1 + 2 C_2 ) +  (- D_{-1} + D_1 + 2 D_2) \geqslant n \label{eq:tn"3} \\
		&  (-C_{-1} + C_1 + 2 C_2 ) + 3 (- D_{-1} + D_1 + 2 D_2) \geqslant n. \label{eq:tn"4}
	\end{align}
	
	Now, multiplying 
	\begin{itemize}
		\item \ref{sum_tnsec_c} by -$\frac{4}{15}$
		\item \ref{eq:tn"1} by $\frac{1}{5}$
		\item \ref{eq:tn"2} by $\frac{1}{5}$
		\item \ref{eq:tn"4} by $\frac{1}{3}$
	\end{itemize} 
	
	and summing up all of them, we get that for any SRD function   $\overline{f}$ 
	\begin{align*}
		\overline{f}(V(T_n'')) \geqslant -\frac{4n}{15} + \frac{n}{5} + \frac{n}{5} + \frac{n}{3} = \frac{7n}{15}
	\end{align*}
	holds, which concludes our proof. 
\end{proof}
	
\begin{rem}
	We applied total enumeration for $T_5''$  and obtained $\gamma_{sR}(T_5'')=5$ which represents the upper bound.
\end{rem}	
\begin{rem}
For this type of polytopes  the theoretical lower bound  from  (\ref{eq:lower_srdp2}) gives $\gamma_{sR}(T_n")=-\frac{2}{3}n$ which could not be of direct usage in proving the above theorem. 
\end{rem}

	\section{Conclusions}
	In this paper we consider the Signed (Total) Domination (S(T)RD) Problem on various classes of planar graphs. On the graph classes $A_n$ and $R_n$ we were able to obtain the exact value of the SRD number and for the graph classes $S_n$ and $T_n$ the exact concerning the STRD problem. On some other graph classes we have proven the exact bounds on the S(T)RD numbers. For example, on graphs $T_n$, $Q_n$, and $T_n''$ for SRD problem and 
	$Q_n$ concerning STRD problem. 
	
	For future work, one could observe the S(T)RD problems on the other classes of graphs (known by the name polytopes). Improving the provided bounds for which the exact values of $\gamma_{sR}$ and $\gamma_{stR}$ are not known is also a reasonable direction of our research. 
	\bibliographystyle{abbrv}	
	\bibliography{bib}	
	
\end{document}